\documentclass[11pt,letterpaper]{amsart}

\usepackage{cases}
\usepackage{txfonts}
\usepackage{mathrsfs}
\usepackage{amsmath}
\usepackage{amssymb}
\usepackage{amsthm}
\usepackage{verbatim}
\usepackage{latexsym, bm}
\usepackage{amscd}
\usepackage{indentfirst}
\usepackage[all]{xy}
\numberwithin{equation}{section}

\title[Pinched theorem]{Pinched theorem and the reverse Yau's inequalities for compact K\"{a}hler-Einstein manifolds}

\author[Rong Du]{Rong Du$^{\dag}$}
\address{Department of Mathematics\\
Key Laboratory of MEA (Ministry of Education) and Shanghai Key Laboratory of PMMP\\
East China Normal University\\
Rm. 312, Math. Bldg, No. 500, Dongchuan Road\\
Shanghai, 200241, P. R. China} \email{rdu@math.ecnu.edu.cn}

\thanks{$^{\dag}$ The research is supported by National Natural Science Foundation of China (Grant No. 12471040), Innovation Action Plan (Basic research projects) of Science and Technology Commission of Shanghai Municipality (Grant No. 21JC1401900), Natural Science Foundation of Chongqing Municipality, China (general program, Grant No. CSTB2023NSCQ-MSX0334) and Science and Technology Commission of Shanghai Municipality (Grant No. 22DZ2229014).}

\theoremstyle{definition}
\newtheorem{theorem}[subsection]{Theorem}
\newtheorem{lemma}[subsection]{Lemma}
\newtheorem{definition}[subsection]{Definition}

\newtheorem{example}{Example}[section]
\newtheorem{corollary}[subsection]{Corollary}
\newtheorem{remark}[subsection]{Remark}

\let\tilde=\widetilde
\allowdisplaybreaks

\begin{document}

\begin{abstract}
For a compact K\"{a}hler-Einstein manifold $M$ of dimension $n\ge 2$, we explicitly write the expression $-c_1^n(M)+\frac{2(n+1)}{n}c_2(M)c_1^{n-2}(M)$ in the form of the integral on a function invloving the holomorphic sectional curvature alone by using the invariant theory. As applications, we get a reverse Yau's inequality and improve the classical $\frac{1}{4}$-pinched theorem and negative $\frac{1}{4}$-pinched theorem for compact K\"{a}hler-Einstein manifolds to smaller pinching constant depending only on the dimension and the first Chern class of $M$. If $M$ is not with positive or negative holomorphic sectional curvature, we characterise the $n$-dimensional complex torus by certain numerical condition. Moreover, we confirm Yau's conjecture for positive holomorphic sectional curvature and Siu-Yang's conjecture for negative holomorphic sectional curvature even for higher dimensions if the absolute value of the holomorphic sectional curvature is small enough. Finally, using the reverse Yau's inequality, we can construct a new example of projective manifold of dimension $n$ $(n\ge 2)$ which is with ample canonical bundle, but does not carry any hermitian metric with negative holomorphic sectional curvature.
\end{abstract}
\maketitle

\vspace{1cm}
\section{\textbf{Introduction}}
Let $(M, J, g)$ be a compact complex manifold of dimension $n\ge 2$ with complex structure $J$ and
K\"{a}hler metric $g$. One can define the sectional curvature
\[K(p)=\frac{R(X,Y,X,Y)}{||X||^2||Y||^2-g(X,Y)^2}\] for real $2$-plane $p$ spanned by $X$ and $Y$ in the tangent space at some point $x\in M$, where $R$ is the Riemannian curvature tensor.

Similarly, if the real $2$-plane $p$ spanned by $X$ and $JX$ is $J$-invariant, then one can also define the holomorphic
sectional curvature
\[H(p)=\frac{R(X,JX,X,JX)}{||X||^4}.\]
In terms of complex coordinates, it is equivalent to write
\[H(p)=R(\xi,\bar{\xi},\xi,\bar{\xi}),\] where $\xi$ is of unit length in the tangent space $T_x(M)$ at $x\in M$.  We denote $$R_{\xi\bar{\xi}\xi\bar{\xi}}:=R(\xi,\bar{\xi},\xi,\bar{\xi})$$ in the context.

Let $\delta$ be a positive number such that $0<\delta\le 1$ and $p$ be a real plane in the tangent space $T_x(M)$ at some point $x\in M$. An $n$-dimensional Riemannian manifold $M$ is said to be $\delta$-pinched at $x$ if there exist a positive constant $K_m$ such that its sectional curvature $K(p)$ satisfies
\begin{equation}\label{p}
\delta K_m\le K(p)\le K_m
\end{equation} for all planes $p\subseteq T_x(M)$. If (\ref{p}) holds for all $x\in M$ and for all planes $p\subseteq T_x(M)$, then $M$ is called $\delta$-pinched. The $\delta$ is called the pinching constant at $x$ and pinching constant respectively.
We say that $M$ is holomorphically $\delta$-pinched at $x$ if there exist a positive constant $H_m$ such that  its holomorphic sectional curvature $H(p)$ satisfies
\begin{equation}\label{hp}
\delta H_m\le H(p)\le H_m
\end{equation}  for all planes $p\subseteq T_x(M)$. If (\ref{hp}) holds for all $x\in M$ and for all planes $p\subseteq T_x(M)$, then $M$ is called holomorphically $\delta$-pinched. The $\delta$ is called the holomorphic pinching constant at $x$ and holomorphic pinching constant respectively.

Similarly, an $n$-dimensional Riemannian manifold $M$ is said to be negative $\delta$-pinched (at $x$) (resp. holomorphically $\delta$-pinched (at $x$)) if $-K(p)$ (resp. $-H(p)$) satisfies (\ref{p}) (resp. (\ref{hp})).

It is known that there is a relation between the concepts of ``$\delta$-pinched" and ``holomorphically $\delta$-pinched" (see \cite{Ko-No} page 369). If $M$ is (negative) holomorphically $\delta$-pinched, then it is (negative) $\frac{1}{4}(3\delta-2)$-pinched. If $M$ is (negative) $\delta$-pinched, then it is (negative) holomorphically $\frac{\delta(8\delta+1)}{1-\delta}$-pinched (see \cite{Ber1} Proposition 1 and \cite{Bi-Go} Section 8). It is obviously that if $\delta<1$ then $\frac{1}{4}(3\delta-2)<\frac{1}{4}$.
We also recommend \cite{Dr-Se} for the equivalent computation.

Now suppose that the K\"{a}hler manifold $M$ is compact. Then $M$ is holomorphically isometric with complex projective space $\mathbb{P}^n$ with a canonical metric if $M$ is $\frac{1}{4}$-pinched or holomorphically $1$-pinched (see \cite{Ko-No} page 369). A K\"{a}hler manifold $M$ with positive holomorphic sectional curvature is less understood. Yau (\cite{Ya3}) asked if the positivity of holomorphic sectional curvature can be used to characterize the rationality of algebraic manifolds. In 1975, Hitchin (\cite{Hit}) proved that any compact K\"{a}hler surface with positive holomorphic sectional curvature must be a rational surface. In higher dimensions, Heier-Wong (\cite{He-Wo}) show that any projective manifold which admits a K\"{a}hler metric with positive holomorphic sectional curvature must be rationally connected. Recently, Yang-Zheng (\cite{Yan-Zh}) confirm Yau's conjecture by adding the condition that the pinching constant is great than or equal to $\frac{1}{2}$.

On the other hand, if a compact K\"{a}hler manifold $M$ is negative $\frac{1}{4}$-pinched or negative holomorphically $1$-pinched, then $M$ is holomorphically isometric to a compact smooth ball quotient (see \cite{Ya-Zh}).
If a K\"{a}hler manifold $M$ is with negative holomorphic sectional curvature, one may wonder whether there is a corresponding curvature characterization for compact
quotients of the complex ball which is the noncompact counterpart of the complex
projective space. However, Mostow-Siu (\cite{Mo-Si}) constructed a compact K\"{a}hler surface which has negative sectional curvature and yet whose universal covering is not biholomorphic to the complex $2$-ball. Therefore, one has to impose additional conditions such as K\"{a}hler-Einstein condition. In \cite{Si-Ya}, the authors conjecture that every K\"{a}hler-Einstein compact complex manifold of complex dimension two with negative sectional curvature is biholomorphic to a compact quotient of the complex $2$-ball. Moreover, they confirm this conjecture by imposing a local pinching condition.

Although the $\frac{1}{4}$-pinched theorem and negative $\frac{1}{4}$-pinched theorem are both classical results, it is a little bit surprising that the pinching constant is independent of the dimension of $M$. If $M$ is with K\"{a}hler-Einstein metric, one could expect that the pinching constant should be improved if the dimension of $M$ is involved. In \cite{Ber2}, Berger shows that if
\[\delta>\frac{(n-2)\Delta}{n-1}>0,\]
then the sectional curvature is everywhere constant, where $\delta$ is the minimum and $\Delta$ is the maximum sectional curvature at points of $M$.

Given a compact K\"{a}hler-Einstein manifold $(M^n, g)$ of dimension $n$ with Ricci curvature $\mu\neq0$, the well-known
computation on the first two Chern numbers is due to Lascoux and Berger (see for example Pages 225-226 in Zheng's famous book \cite{zhe}). Our method is using the invariant theory to write
\[\frac{2(n+1)}{n}c_1^{n-2}c_2-c_1^n\] in the form of the integral of the variance of the holomorphic sectional curvatures at each point in $M$.
As a result, we improve the classical $\frac{1}{4}$-pinched theorem and negative $\frac{1}{4}$-pinched theorem for compact K\"{a}hler-Einstein manifold to some smaller pinching constant depending only on the dimension and the first Chern class of $M$.

Suppose that $V=\text{vol}(\mathbb{S}^{2n-1})$ is the volume of the unit sphere $\mathbb{S}^{2n-1}$ of dimension $2n-1$ and \[\text{Ave}(R):=\frac{1}{V}\int_{\xi\in
\mathbb{S}^{2n-1}\subseteq T_x(M)}
R_{\xi\bar{\xi}\xi\bar{\xi}}~d\mathbb{S}^{2n-1}\] is the average value of the holomorphic sectional curvature at $x$, which by Berger's lemma is proportional to the scalar curvature, namely, $\text{Ave}(R)=\frac{2\mu}{n+1}$. We also can get the result from (\ref{AveR}) and (\ref{Berger}).

Now, let $f$ be the function on $M$ given by $$f(x)=\frac{1}{V}\int_{\xi\in
\mathbb{S}^{2n-1}\subseteq T_x(M)} (R_{\xi\bar{\xi}\xi\bar{\xi}}-Ave(R))^2 d\mathbb{S}^{2n-1}.$$
\begin{theorem}\label{formula}
Let $(M^n, g)$ be a compact K\"{a}hler-Einstein manifold of dimension $n$ with Ricci curvature $\mu\neq 0$. Then
\begin{equation}\label{conj}
\frac{2(n+1)}{n}c_1^{n-2}c_2-c_1^{n}=D_n\int_M f d\tilde{v},
\end{equation}
where $$D_n=\frac{(n+1)\cdot(n+3)!\mu^{n-2}}{4(\pi)^{n}n(n-1)}$$ and $$d\tilde{v}=\left(\frac{\textbf{i}}{2}\right)^ndz_{1}\wedge d\bar{z}_{1}\wedge\cdots\wedge dz_{n}\wedge d\bar{z}_{n}.$$
\end{theorem}

As a corollary, we can get the following several theorems.

First, we have reverse Yau's inequality.
\begin{theorem}\label{rY}
Let $(M^n, g)$ be a compact K\"{a}hler-Einstein manifold of dimension $n$ with Ricci curvature $\mu\neq 0$.  If the holomorphic sectional curvature is nonpositive or nonnegative, then
\begin{equation}
2(n+1)c_2(M)|c_1^{n-2}(M)| < \left(\frac{(n+1)^2(n+2)(n+3)H_m^2}{16(n-1)}+n\right)|c_1^n(M)|.
\end{equation}
In particular,
\begin{equation}
2(n+1)c_2(M)|c_1^{n-2}(M)| < \left(\frac{(n+1)^2(n+2)(n+3)}{16(n-1)}+n\right)|c_1^n(M)|.
\end{equation}
\end{theorem}

\begin{remark}
The author and Sun (\cite{Du-Su}) prove that $c_2(M)|c_1^{n-2}(M)|$ is controlled by $|c_1^{n}(M)|$ up to a constant if $M$ is projective with ample canonical bundle.  Recently, Li and Zheng get a reverse Yau's inequality (\cite{Li-Zh}), but they need add the condition that canonical divisor is globally generated.

Moreover, using this reverse Yau's inequality, we can construct a projective manifold $M$ of any dimension $n$ $(n\ge 2)$ which is with ample canonical bundle $K_M$, but doesn't carry any hermitian metric with negative holomorphic sectional curvature (Example \ref{example dimn}).
\end{remark}

\begin{theorem}\label{ampleK}
Let $(M^n, g)$ be a compact K\"{a}hler-Einstein manifold of dimension $n$ with with Ricci curvature $\mu\neq 0$. If there exists a positive constant $H_m$, such that the holomorphic sectional curvature $H(p)$ satisfies $-H_m\le H(p)\le H_m$ for all $p$ and $x\in M$. Then we have
\begin{equation}
2(n+1)c_2(M)|c_1^{n-2}(M)| <  \left(\frac{(n+1)^2(n+2)(n+3)H_m^2}{4(n-1)}+n\right)|c_1^n(M)|
\end{equation}
for $n\ge 2$.
\end{theorem}

\begin{theorem}\label{pinch}
Let $(M^n, g)$ be a compact K\"{a}hler-Einstein manifold of dimension $n$ with Ricci curvature $\mu\neq 0$.
\begin{enumerate}
\item [1)] Suppose that $M$ is with positive holomorphic sectional curvature and $\mu>0$. If $M$ is holomorphically $\delta$-pinched, where
\[\delta>1-\left(\frac{4(n-1)}{c_1^n(M)(n+1)^{2}(n+2)(n+3)}\right)^{1/2},\] then $M$ is holomorphically isometric to $\mathbb{P}^n$.
\item [2)] Suppose that $M$ is with negative holomorphic sectional curvature and $\mu<0$. If $M$ is negative holomorphically $\delta$-pinched, where \[\delta>1-\left(\frac{4(n-1)}{|c_1^n(M)|(n+1)^{2}(n+2)(n+3)}\right)^{1/2},\] then $M$ is holomorphically isometric to a ball quotient.
\item [3)] Suppose that $M$ is not with positive or negative holomorphic sectional curvature. If $-\delta_1\le R_{\xi\bar{\xi}\xi\bar{\xi}}\le \delta_2$ $(\delta_1, \delta_2\ge0)$ for all $\xi\in \mathbb{S}^{2n-1}\subseteq T_x(M)$ and for all $x\in M$.  Let $a=max\{\delta_2-Ave(R), Ave(R)+\delta_1\}$, and \[a<\frac{2|\mu|}{n+1}\cdot\left(\frac{n-1}{|c_1(M)^{n}|(n+2)(n+3)}\right)^{\frac{1}{2}},\] then $M$ is a holomorphically isometric to a complex torus.
\end{enumerate}
\end{theorem}

Using Theorem \ref{formula}, we can also confirm Yau's conjecture and Siu-Yang's conjecture for small absolute value of the holomorphic sectional curvatures at some point in $M$.

\begin{theorem}\label{small}
Let $(M^n, g)$ be a compact K\"{a}hler-Einstein manifold with Ricci curvature $\mu\neq 0$.
\begin{enumerate}
\item [1)] Suppose that $M$ is with positive holomorphic sectional curvature. If
\[0<H(p)\le\frac{4}{n+1}\left(\frac{n-1}{c_1^{n}(M)(n+2)(n+3)}\right)^{\frac{1}{2}}\] for all $x\in M$ and for all planes $p\subseteq T_x(M)$ invariant by $J$, then $M$ is holomorphically isometric to $\mathbb{P}^n$.
\item [2)] Suppose that $M$ is with negative holomorphic sectional curvature.  If  \[-\frac{4}{n+1}\left(\frac{n-1}{|c_1^{n}(M)|(n+2)(n+3)}\right)^{\frac{1}{2}}\le H(p)<0\] for all $x\in M$ and for all planes $p\subseteq T_x(M)$ invariant by $J$, then $M$ is holomorphically isometric to a ball quotient.
\end{enumerate}
\end{theorem}

\begin{remark}
If $M$ is K\"{a}hler-Einstein with positive holomorphic sectional curvature, then by Yang-Zheng's result (\cite{Yan-Zh}), Yau's conjecture is still open in general for
\[\frac{4}{n+1}\left(\frac{n-1}{c_1^{n}(M)(n+2)(n+3)}\right)^{\frac{1}{2}}<\delta<\frac{1}{2}.\]
If $M$ is K\"{a}hler-Einstein with negative holomorphic sectional curvature, Siu and Yang (\cite{Si-Ya}) impose the condition that for every point $x$ of the
manifold the average holomorphic sectional curvature at $x$ cannot be too close to
the maximum holomorphic sectional curvature at $x$, in relation to the minimum holomorphic sectional curvature at $x$. However, our conclusion is just a part result that their theorem cannot cover.

\end{remark}

\section{\textbf{Spherical Harmonics and Integration over the Sphere}}
For calculating the coefficient $D$ in Theorem \ref{formula}, we evaluate at two examples, say $\mathbb{CP}^n$ and ${(\mathbb{CP}^1)}^{\times n}$,
so we need to calculate integration over the unit sphere (Note that similar integral is also considered by Hall-Murphy in \cite{Ha-Mu} for different purpose).  We recommend
\cite{At-Ha} for the reference.

Let $S^n:=\mathbb{C}[x_1,x_2,\cdots,
x_n]$ be the polynomial ring of $n$ variables and $S^n_k$ be the set
of all homogeneous polynomials of degree $k$ of $S^n$. Then
$S^n=\oplus_{k=1}^{+\infty} S^n_k$. We also recall the notation of
double factorial,

\begin{equation} \label{!!}
m!!=\left\{
\begin{aligned}
         &m(m-2)(m-4)\cdots 2, & m ~\text{even} \\
                  &m(m-2)(m-4)\cdots 1, &m
                  ~\text{odd}.
                          \end{aligned} \right.
                          \end{equation}
In particular, we define $0!!=1$.

It is convenient to use the multi-index notation. A multi-index with
$n$ components is ${\bm{\alpha}}=(\alpha_1, \cdots, \alpha_n)$,
where $\alpha_1, \cdots, \alpha_n$ are non-negative integers. The
length of $\bm{\alpha}$ is $|\bm{\alpha}|=\sum_{j=1}^n \alpha_j$. We
write ${\bm{\alpha}}!$ to mean $\alpha_1 !\cdots \alpha_n !$. With
$\bm{x}=(x_1, \cdots, x_n)^T$ we define
\[{\bm{x}}^{\bm{\alpha}}:=x_1^{\alpha_1}\cdots x_n^{\alpha_n}.\]
Similarly, with the gradient operator
${\bm{\nabla}}=(\partial_{x_1},\cdots, \partial_{x_n})^T$, we define
\[{\bm{\nabla}}^{\bm{\alpha}}:=\frac{\partial^{|\alpha|}}{\partial x_1^{\alpha_1}\cdots \partial x_n^{\alpha_n}}.\]

Denote
\[\Delta_{(n)}:=\sum_{j=1}^n \frac{\partial^2}{\partial x_j ^2}\]to be the Laplacian operator.
We will write $\Delta$ instead of $\Delta_{(n)}$ if there is no
confusion in the context.

Any $f_k\in S^n_k$ can be written in the form
\[f_k({\bm{x}})=\sum_{|{\bm{\alpha}}|=n}a_{\bm{\alpha}}{\bm{x}}^{\bm{\alpha}}\quad a_{\bm{\alpha}}\in \mathbb{C}.\]
For this polynomial $f_k$, define
\[f_k({\bm{\nabla}})=\sum_{|{\bm{\alpha}}|=n}a_{\bm{\alpha}}{\bm{\nabla}}^{\bm{\alpha}}.\]
Given any two polynomials in $S^n_k$,
\[f_k({\bm{x}})=\sum_{|{\bm{\alpha}}|=n}a_{\bm{\alpha}}{\bm{x}}^{\bm{\alpha}},\quad g_k({\bm{x}})=\sum_{|{\bm{\alpha}}|=n}b_{\bm{\alpha}}{\bm{x}}^{\bm{\alpha}},\]
it is straightforward to show
\[f_k({\bm{\nabla}})\overline{g_k({\bm{x}})}=\sum_{|{\bm{\alpha}}|=n} {\bm{\alpha}}! a_{\bm{\alpha}}\overline{b_{\bm{\alpha}}}=\overline{g_k({\bm{\nabla}})\overline{f_k({\bm{x}})}}.\] So
\begin{equation}\label{inner}
(f_k, g_k)_{S^n_k}:=f_k({\bm{\nabla}})\overline{g_k({\bm{x}})}
\end{equation}
defines an inner product in the subspace $S^n_k$.

\begin{definition}
Denote $\mathbb{H}^n_k$ to be the space of all the homogeneous
harmonics of degree $k$ in $n$ variables, i.e., the space of all the
$f_k\in S^n_k$ such that $\Delta(f_k)=0$.
$\mathbb{SH}^n_k$:=$\mathbb{H}^n_k|_{\mathbb{S}^{n-1}}$ is called
the spherical harmonics space of order $k$ in $n$ dimensions.
\end{definition}

\begin{lemma}\label{per}
For $p\neq q$, $\mathbb{SH}^n_p\bot \mathbb{SH}^n_q$ with respect to
the $L^2$ metric on $\mathbb{S}^{n-1}$.
\end{lemma}
In particular, Lemma \ref{per} implies
\begin{equation}\label{sorth}
\int_{\mathbb{S}^{n-1}} h_j({\bm{x}})~d\mathbb{S}^{n-1}=0, \quad
\text{for any}~h_j\in \mathbb{H}^n_j,~ j\ge 1.
\end{equation}
\begin{theorem}(\cite{At-Ha} Theorem 2.18)\label{dec}
With respect to the inner product (\ref{inner}), we have following
decomposition
\begin{equation}\label{dec form}
S^n_k=\mathbb{H}^n_k\oplus
|\cdot|^2\mathbb{H}^{n}_{k-2}\oplus\cdots\oplus|\cdot|^{2[k/2]}\mathbb{H}^{n}_{k-2[k/2]}.
\end{equation}
\end{theorem}
Theorem \ref{dec} means that for $f_k\in S^n_k$, we have
\begin{equation}\label{decf}
f_k({\bm{x}})=\sum_{j=0}^{[k/2]}|{\bm{x}}|^{2j}h_{k-2j}({\bm{x}}),
\end{equation}
where $h_{k}\in \mathbb{H}^n_k$.

Applying the Laplacian operator $\Delta^t$ to both sides of
(\ref{decf}), we have
\begin{equation}\label{tLap}
\Delta^t(f_k({\bm{x}}))=\sum_{j=t}^{[k/2]}\frac{(2j)!!(n+2d-2j-2)!!}{(2j-2t)!!(n+2k-2j-2t-2)!!}|{\bm{x}}|^{2(j-t)}h_{k-2j}({\bm{x}}),
\end{equation}
where $h_{k}\in \mathbb{H}^n_k$. In particular, for $k$ even,
\begin{equation}\label{h0}
\Delta^{k/2}f_k({\bm{x}})=\frac{k!!(n+k-2)!!}{(n-2)!!}h_0({\bm{x}}).
\end{equation}

Let $f_k\in S^n_k$, by (\ref{sorth}), (\ref{decf}) and (\ref{h0}),
we have
\begin{equation} \label{int}
\int_{\mathbb{S}^{n-1}} f_k~d\mathbb{S}^{n-1}=\left\{
\begin{aligned}
         &0&& k ~\text{odd} \\
                  &\frac{(n-2)!! \text{vol}(\mathbb{S}^{n-1})}{k!!(n+k-2)!!}\Delta^{k/2}(f_k)&&k
                  ~\text{even}.
                          \end{aligned} \right.
                          \end{equation}
Note that for $n$ even and $f_k\in \mathbb{H}^n_k$,
$\Delta^{k/2}(f_k)$ is a constant.

\section{\textbf{Invariant theory and a formula of Chern numbers}}
Let $(M^n, g)$ be a compact K\"{a}hler-Einstein manifold of dimension $n$ with Ricci curvature $\mu\neq 0$. Then its Ricci form is equal to its K\"{a}hler form up to a constant. That is,
\[\textbf{i}\sum_{i,j}Ric_{i\bar{j}}dz_i\wedge d\bar{z}_j=\mu \textbf{i}\sum_{i,j} g_{i\bar{j}}(x)dz_i\wedge d\bar{z}_j,\]
where $\textbf{i}=\sqrt{-1}, 1\le i, j\le n$.
Denote $\Theta=(\Theta_{\alpha}^{\beta})$, where $1\le \alpha, \beta\le n$, to be the curvature matrix. Since the first Chern form of $M$ is the Ricci form of the tangent bundle of $M$, we have
\[c_1(\Theta)=\frac{\textbf{i}}{2\pi}\sum_{i,j}Ric_{i\bar{j}}dz_i\wedge d\bar{z}_j=\frac{\textbf{i}}{2\pi}\sum_{i,j} \mu g_{i\bar{j}}(x)dz_i\wedge d\bar{z}_j.\]
For $x\in M$, we can choose the local coordinates such that $g_{i\bar{j}}(x)=\delta_{ij}(x)$. So
\[c_1^n(M)=n!\left(\frac{\mu}{\pi}\right)^n\int_M d\tilde{v},\]
where $d\tilde{v}=\left(\frac{\textbf{i}}{2}\right)^ndz_{1}\wedge d\bar{z}_{1}\wedge\cdots\wedge dz_{n}\wedge d\bar{z}_{n}$.

So the volume of $M$ is
\begin{equation}\label{Vc1}
V(M)=\frac{1}{n!}\left(\frac{\pi}{\mu}\right)^nc_1^n(M)
\end{equation}
if $\mu\neq 0$.

Suppose that $T_x(M)$ is the tangent space of
$M$ at $x$ and $V:=\text{vol}(\mathbb{S}^{2n-1})$ is the volume of the
unit sphere. Denote
\[\text{Ave}(R):=\frac{1}{V}\int_{\xi\in
\mathbb{S}^{2n-1}\subseteq T_x(M)}
R_{\xi\bar{\xi}\xi\bar{\xi}}~d\mathbb{S}^{2n-1}\] to be the average of integral of the holomorphic sectional curvature along the unit sphere in the tangent space of $M$ at some point $x\in M$. We omit
``$\xi\in
\mathbb{S}^{2n-1}\subseteq T_x(M)$ '' under the ``$\int$" and $d\mathbb{S}^{2n-1}$ in
the later context if there is no confusion in the context, for example,
\[\text{Ave}(R):=\frac{1}{V}\int
R_{\xi\bar{\xi}\xi\bar{\xi}}.\]

Let
\begin{equation} \label{inv}
\left\{ \begin{aligned}
         R^{(m)}_1 &=\sum_{\alpha_1,\cdots,\alpha_m;\atop \beta_1\cdots,\beta_m}\prod_{t=1}^m R_{\alpha_t\bar{\alpha_t}\beta_t\bar{\beta_t}},\quad\quad m\ge 1\\
                  R^{(m)}_2&=\sum_{i,j,k,l;\atop \alpha_1,\cdots,\alpha_{m-2};\beta_1\cdots,\beta_{m-2}} |R_{i\bar{j}k\bar{l}}|^2
                  \prod_{t=1}^{m-2}
                  R_{\alpha_t\bar{\alpha_t}\beta_t\bar{\beta_t}},\quad\quad m\ge 2.
                          \end{aligned} \right.
                          \end{equation}
where $$1\le i,j,k,l, \alpha_1,\cdots, \alpha_m,\beta_1,\cdots,\beta_m\le n.$$

The following theorem is known (see \cite{Bes} page 133) which comes from Weyl's invariant theory on classical groups (\cite{Wey}).

\begin{theorem}
Any polynomial of degree $2$ in the curvature tensors which is invariant under unitary transformation can be represented in the three basis with coefficients in $\mathbb{R}$:
\[\{\sum_{i,j,k,l}R_{i\bar{i}j\bar{j}}R_{k\bar{k}l\bar{l}},\sum_{i,j,k,l}|R_{i\bar{j}k\bar{l}}|^2, \sum_{i,j}|Ric_{ij}|^2\} .\]
Any polynomial of degree $1$ in the curvature tensors which is invariant under unitary transformation can be represented in the two basis with coefficients in $\mathbb{R}$:
\[\{\sum_{i,j}R_{i\bar{i}j\bar{j}}, \sum_{i,j}Ric_{ij}\} .\]
\end{theorem}

\begin{corollary}\label{invariant}
If K\"{a}hler-Einstein condition satisfies, then any polynomial of degree $2$ in the curvature tensors which is invariant under unitary transformation can be represented in the two basis with coefficients in $\mathbb{R}$:
\[\{R_1^{(2)},R_2^{(2)}\}\] and any polynomial of degree $1$ in the curvature tensors which is invariant under unitary transformation can be represented in the one basis with coefficients in $\mathbb{R}$:
\[\{R_1^{(1)}\}.\]
\end{corollary}
\begin{proof}
Use the condition of K\"{a}hler-Einstein and choose the local coordinates such that $g_{i\bar{j}}(x)=\delta_{ij}(x)$ for $x\in M$, then
\[\sum_{i,j}|Ric_{ij}|^2=\sum_{i,j}\mu^2|g_{i\bar{j}}|^2=\mu^2n,\]
\[\sum_{i,j,k,l}R_{i\bar{i}j\bar{j}}R_{k\bar{k}l\bar{l}}=\sum_{i,j}Ric_{i\bar{i}}Ric_{j\bar{j}}=\mu^2g_{i\bar{i}}g_{j\bar{j}}=\mu^2n^2.\]
Thus,
\[\sum_{i,j}|Ric_{ij}|^2=\frac{1}{n}\sum_{i,j,k,l}R_{i\bar{i}j\bar{j}}R_{k\bar{k}l\bar{l}}.\]
Similarly,
\[\sum_{i,j}Ric_{ij}=\sum_{i,j}\mu g_{i\bar{j}}=\mu n,\]
\[\sum_{i,j}R_{i\bar{i}j\bar{j}}=\sum_{i}Ric_{i\bar{i}}=\mu \sum_{i}g_{i\bar{i}}=\mu n.\]
Thus,
\begin{equation}\label{Ric=R}
\sum_{i,j}Ric_{ij}=\sum_{i,j}R_{i\bar{i}j\bar{j}}.
\end{equation}
\end{proof}

Since $\int R_{\xi\bar{\xi}\xi\bar{\xi}}^2$ and $\int
R_{\xi\bar{\xi}\xi\bar{\xi}}$ are invariant under unitary
transformation, by Corollary \ref{invariant}, we can assume
\begin{equation}\label{deg2com}
\int R_{\xi\bar{\xi}\xi\bar{\xi}}^2=A\cdot R^{(2)}_1+B\cdot
R^{(2)}_2
\end{equation}
and
\begin{equation}\label{deg2com}
\int R_{\xi\bar{\xi}\xi\bar{\xi}}=C\cdot R^{(1)}_1,
\end{equation}
where $A, B, C \in \mathbb{R}$.

We will get these coefficients by evaluating at two examples, say $\mathbb{CP}^n$ and ${(\mathbb{CP}^1)}^{\times n}$.
Let $\xi\in\mathbb{S}^{2n-1}=\{(z_1,\cdots, z_n)|\sum_i |z_i|^2=1\}$. Assume $g$ to be the K\"{a}hler-Einstein metric. For $x\in M$, we can choose the local coordinates such that $g_{i\bar{j}}(x)=\delta_{ij}(x)$.

Case I: $\mathbb{CP}^n$.

Let $g_{FS}=\sum g_{i\bar{j}} dz_i\otimes d\bar{z}_j$ be the Fubini-Study metric . At the point $x$, we have
\[R_{i\bar{j}k\bar{l}}=g_{i\bar{j}}g_{k\bar{l}}+g_{i\bar{l}}g_{k\bar{j}}=\delta_{ij}\delta_{kl}+\delta_{il}\delta_{kj}\quad i,j,k,l=1,2,\cdots n.\]
So \[R^{(1)}_1=n(n+1)\quad R^{(2)}_1=n^2(n+1)^2, \quad R^{(2)}_2=2n(n+1).\]
On the other hand, \[\int R_{\xi\bar{\xi}\xi\bar{\xi}}=2\int (\sum_{i=1}^n |z_i|^2)^2=2V\]
and
\[\int R_{\xi\bar{\xi}\xi\bar{\xi}}^2=4\int (\sum_{i=1}^n |z_i|^2)^4=4V.\]
So $\text{Ave}(R)=2$ and $C=\frac{2}{n(n+1)}$.

Therefore, in general,
\begin{equation}\label{AveR}
\text{Ave}(R)=\frac{2}{n(n+1)}R^{(1)}_1
\end{equation}
and
\begin{equation}\label{equ1}
n^2(n+1)^2A+2n(n+1)B=4V.
\end{equation}

Case II: ${(\mathbb{CP}^1)}^{\times n}$.

Let $g$ be the $n$ times product metric of Fubini-Study metric of $\mathbb{P}^1$.
Then
\[R_{i\bar{i}i\bar{i}}=2~ \text{for all}~ i=1,2,\cdots, n~ \text{and others are}~ 0.\]
So \[R^{(2)}_1=4n^2, \quad R^{(2)}_2=4n.\]
On the other hand, suppose $z_i=x_i+y_i\mathbf{i}$, then by (\ref{int})

\begin{equation} \label{}
\begin{split}
\int R_{\xi\bar{\xi}\xi\bar{\xi}}^2&=4\int (\sum_{i=1}^n |z_i|^4)^2\\
 &= 4\int (\sum_{i=1}^n (x_i^2+y_i^2)^2)^2\\
 &= 4\frac{(2n-2)!!V}{8!!(2n+8-2)!!}\Delta^4(H_8),
 \end{split}
 \end{equation}
where $H_8:=(\sum_{i=1}^n (x_i^2+y_i^2)^2)^2$.

Calculating directly, we get
\begin{equation} \label{}
\begin{split}
\Delta(H_8)&=\Delta((\sum_{i=1}^n (x_i^2+y_i^2)^2)^2);\\
 &= 32\sum_{i=1}^n (x_i^2+y_i^2)^2\cdot \sum_{i=1}^n (x_i^2+y_i^2)+32\sum_{i=1}^n (x_i^2+y_i^2)^3;\\
\Delta^2(H_8)&= 32(16\sum_{i=1}^n (x_i^2+y_i^2)^2+(4n+52)\sum_{i=1}^n (x_i^2+y_i^2)^2);\\
\Delta^3(H_8)&= 32(64(3n+15)\sum_{i=1}^n (x_i^2+y_i^2));\\
\Delta^4(H_8)&=2^{13}n(3n+15).
 \end{split}
 \end{equation}
So
\begin{equation} \label{}
\begin{split}
\int R_{\xi\bar{\xi}\xi\bar{\xi}}^2&=4\frac{(2n-2)!!V}{8!!(2n+8-2)!!}\cdot 2^{13}n(3n+15)\\
 &= \frac{16(n+5)V}{(n+1)(n+2)(n+3)}.
 \end{split}
 \end{equation}
Combining equation (\ref{equ1}), we have
 \begin{equation} \label{}
\left\{ \begin{aligned}
         n^2(n+1)^2A&+2n(n+1)B &=&4V\\
                  4n^2A&+4nB&=& \frac{16(n+5)V}{(n+1)(n+2)(n+3)},
                          \end{aligned} \right.
                          \end{equation}
and get

 \begin{equation} \label{}
\left\{ \begin{aligned}
        A&=\frac{4(n+4)V}{n^2(n+1)(n+2)(n+3)}\\
                 B&= \frac{4V}{n(n+1)(n+2)(n+3)}.
                          \end{aligned} \right.
                          \end{equation}

So
\begin{equation} \label{}
\begin{split}
&\int (R_{\xi\bar{\xi}\xi\bar{\xi}}-Ave(R))^2\\
=&\frac{4(n+4)V}{n^2(n+1)(n+2)(n+3)}R_1^{(2)}+ \frac{4V}{n(n+1)(n+2)(n+3)}R_2^{(2)}-\frac{4V}{n^2(n+1)^2}R_1^{(2)}\\
=&4V\left(\frac{-2}{n^2(n+1)^2(n+2)(n+3)}R_1^{(2)}+\frac{1}{n(n+1)(n+2)(n+3)}R_2^{(2)}\right)
 \end{split}
 \end{equation}
and
 \begin{equation} \label{biggerthan0}
\begin{split}
&\int (R_{\xi\bar{\xi}\xi\bar{\xi}}-Ave(R))^2 Ave(R)^{n-2}\\
=&4V\frac{2^{n-2}}{n^{n-2}(n+1)^{n-2}}\left(\frac{-2}{n^2(n+1)^2(n+2)(n+3)}R_1^{(n)}+\frac{1}{n(n+1)(n+2)(n+3)}R_2^{(n)}\right)\\
=&\frac{2^{n}V}{n^{n}(n+1)^{n}(n+2)(n+3)}(-2R_1^{(n)}+n(n+1)R_2^{(n)}).
 \end{split}
 \end{equation}

Let $g$ be the K\"{a}hler-Einstein metric. For $x\in M$, we can choose the local coordinates such that $g_{i\bar{j}}(x)=\delta_{ij}(x)$. Denote $\Theta=(\Theta_{\alpha}^{\beta})$ to be the curvature matrix where $\Theta_{\alpha}^{\beta}=\sum \Theta_{\alpha,i,j}^{\beta}dz_i\wedge d\bar{z}_j$ and $R=\sum R_{i\bar{j}k\bar{l}}dz_i\otimes d\bar{z}_j\otimes dz_k\otimes d\bar{z}_l$ to be the curvature tensor.
We have
\begin{equation}\label{thetaR}
R_{i\bar{j}k\bar{l}}=\sum_p g_{pj}\Theta^p_{ik\bar{l}}=\Theta^j_{ik\bar{l}}.
\end{equation}

Let $\sigma$ be a permutation in symmetric group $S_n$ and let $m_1, \cdots, m_s$ be the distinct integers which appear in the cycle type of $\sigma$ (including 1-cycle). For each $i\in \{1,2,\cdots, n\}$, assume $\sigma$ has $k_i$ cycles of length $m_i$. So $\sum_{i=1}^s k_im_i=n$. Denote $g(\{(m_i,k_i); n\}_{i=1}^s)$ to be the number of conjugates of $\sigma$. The following result is well known.

\begin{lemma}\label{cn}
\begin{equation}
g(\{(m_i,k_i)\}_{i=1}^s;n)=\frac{n!}{\prod_{i=1}^sk_i!m_i^{k_i}}.
\end{equation}
\end{lemma}

\begin{definition}
A partition of positive integer $n$ is a finite decreasing sequence $\lambda$ of natural number  $\lambda_1\ge \lambda_2\ge \cdots\ge \lambda_{l(\lambda)}>0$ and the weight $|\lambda|:=\sum_{i=1}^{l(\lambda)}\lambda_i=n$. The $l(\lambda)$ is called the length of $\lambda$.
\end{definition}

In order to calculate the first and the second Chern forms in terms of $R_1^{(n)}$ and $R_2^{(n)}$, the crucial steps are following identities.
\begin{equation*}
\begin{split}
\sum_{\alpha_1,\dots, \alpha_n}\sum_{\sigma\in S_n} \prod_{i=1}^n R_{\alpha_i\bar{\alpha_i}i\overline{{\sigma(i)}}}
=&\sum_{\sigma\in S_n} \prod_{i=1}^n Ric_{i\overline{{\sigma(i)}}}\\
=&\mu^n\sum_{\sigma\in S_n} \prod_{i=1}^n \delta_{i\overline{{\sigma(i)}}}\\
=&\mu^n\sum_{|\lambda|=n}g(\{(m_i,k_i)\}_{i=1}^s;n)n^{l(\lambda)}
\end{split}
\end{equation*}
and
\begin{equation}\label{Berger}
\begin{split}
R^{(n)}_1 &=\sum_{\alpha_1,\cdots,\alpha_n;\atop \beta_1,\cdots,\beta_n}\prod_{t=1}^n R_{\alpha_t\bar{\alpha_t}\beta_t\bar{\beta_t}}\\
&=\sum_{\beta_1,\cdots,\beta_n}\prod_{t=1}^n Ric_{\beta_t\bar{\beta_t}}\\
&=\mu^n \sum_{\beta_1,\cdots,\beta_n}\prod_{t=1}^n \delta_{\beta_t\bar{\beta_t}}\\
&=\mu^nn^n.
\end{split}
\end{equation}

Thus,
\begin{equation}\label{RR1}
\sum_{\alpha_1,\dots, \alpha_n}\sum_{\sigma\in S_n} \prod_{i=1}^n R_{\alpha_i\bar{\alpha_i}i\overline{{\sigma(i)}}}=\sum_{|\lambda|=n}g(\{(m_i,k_i)\}_{i=1}^s;n)\frac{R^{(n)}_1}{n^{n-l(\lambda)}}.
\end{equation}

\begin{equation*}
\begin{split}
&\sum_{\alpha_1,\dots, \alpha_n}\sum_{\sigma\in S_n} \prod_{i=1}^{n-2} R_{\alpha_i\bar{\alpha_i}i\overline{{\sigma(i)}}}R_{\alpha_{n-1}\bar{\alpha}_nn-1\overline{{\sigma(n-1)}}}R_{\alpha_{n}\bar{\alpha}_{n-1}n\overline{{\sigma(n)}}}\\
=&\sum_{\alpha_{n-1}, \alpha_n}\sum_{\sigma\in S_n} \prod_{i=1}^{n-2} Ric_{i\overline{{\sigma(i)}}}R_{\alpha_{n-1}\bar{\alpha}_nn-1\overline{{\sigma(n-1)}}}R_{\alpha_{n}\bar{\alpha}_{n-1}n\overline{{\sigma(n)}}}\\
=&\mu^{n-2}\sum_{\alpha_{n-1}, \alpha_n}\sum_{\sigma\in S_n} \prod_{i=1}^{n-2} \delta_{i\overline{{\sigma(i)}}}R_{\alpha_{n-1}\bar{\alpha}_nn-1\overline{{\sigma(n-1)}}}R_{\alpha_{n}\bar{\alpha}_{n-1}n\overline{{\sigma(n)}}}\\
=&\mu^n\sum_{|\lambda|=n}n^{l(\lambda)-1}g(\{(m_i,k_i)\}_{i=1}^s;n-2) +\mu^{n-2}\sum_{|\lambda|=n}n^{l(\lambda)-1}g(\{(m_i,k_i)\}_{i=1}^s;n-2) \sum_{i,j,k,l}|R_{i\bar{j}k\bar{l}}|^2
\end{split}
\end{equation*}
and
\begin{equation*}
\begin{split}
 R^{(n)}_2&=\sum_{i,j,k,l;\atop \alpha_1,\cdots,\alpha_{n-2};\beta_1,\cdots,\beta_{n-2}} |R_{i\bar{j}k\bar{l}}|^2
 \prod_{t=1}^{n-2}R_{\alpha_t\bar{\alpha_t}\beta_t\bar{\beta_t}}\\
 &=\sum_{i,j,k,l;\atop \beta_1,\cdots,\beta_{n-2}} |R_{i\bar{j}k\bar{l}}|^2
 \prod_{t=1}^{n-2}Ric_{\beta_t\bar{\beta_t}}\\
 &=\mu^{n-2}\sum_{i,j,k,l;\atop \beta_1,\cdots,\beta_{n-2}} |R_{i\bar{j}k\bar{l}}|^2
 \prod_{t=1}^{n-2}\delta_{\beta_t\bar{\beta_t}}\\
 &=\mu^{n-2}n^{n-2}\sum_{i,j,k,l} |R_{i\bar{j}k\bar{l}}|^2.
 \end{split}
 \end{equation*}

We set \[G_s:=\{\sigma\in S_n| n~ \text{and}~ n-1 ~\text{are in the same cycle of } ~\sigma\},\]
and \[G_d:=\{\sigma\in S_n| n~ \text{and}~ n-1 ~\text{are in two different cycles of } ~\sigma\}.\]
Thus,
\begin{equation}\label{RR2}
\begin{split}
&\sum_{\alpha_1,\dots, \alpha_n}\sum_{\sigma\in S_n} \prod_{i=1}^{n-2} R_{\alpha_i\bar{\alpha_i}i\overline{{\sigma(i)}}}R_{\alpha_{n-1}\bar{\alpha}_nn-1\overline{{\sigma(n-1)}}}R_{\alpha_{n}\bar{\alpha}_{n-1}n\overline{{\sigma(n)}}}\\
=&\sum_{\alpha_1,\dots, \alpha_n}\sum_{\sigma\in G_d} \prod_{i=1}^{n-2} R_{\alpha_i\bar{\alpha_i}i\overline{{\sigma(i)}}}R_{\alpha_{n-1}\bar{\alpha}_nn-1\overline{{\sigma(n-1)}}}R_{\alpha_{n}\bar{\alpha}_{n-1}n\overline{{\sigma(n)}}}\\
&\mbox \qquad \qquad\qquad+\sum_{\alpha_i,i}\sum_{\sigma\in G_s} \prod_{i=1}^{n-2} R_{\alpha_i\bar{\alpha_i}i\overline{{\sigma(i)}}}R_{\alpha_{n-1}\bar{\alpha}_nn-1\overline{{\sigma(n-1)}}}R_{\alpha_{n}\bar{\alpha}_{n-1}n\overline{{\sigma(n)}}}\\
=&\sum_{|\lambda|=n-\iota\atop \iota_1+\iota_2=\iota}C_{n-2}^{\iota_1-1}C_{n-\iota_1-1}^{\iota_2-1}g(\{(m_i,k_i)\}_{i=1}^s;n-\iota)g(\{(\iota_1,1)\};\iota_1)g(\{(\iota_2,1)\};\iota_2)\frac{R_1^{(n)}}{n^{n-l(\lambda)-1}}\\ &\mbox \qquad \qquad  +\sum_{|\lambda|=n-\iota}C_{n-2}^{\iota_1-2}g(\{(m_i,k_i)\}_{i=1}^s;n-\iota)g(\{(\iota,1)\};\iota)\frac{R_2^{(n)}}{n^{n-l(\lambda)-2}}\\
=&\sum_{|\lambda|=n-\iota}\frac{(n-2)!(\iota-1)}{(n-\iota)!}g(\{(m_i,k_i)\}_{i=1}^s;n-\iota)\frac{R_1^{(n)}}{n^{n-l(\lambda)-1}}\\
&\mbox \qquad \qquad  +\sum_{|\lambda|=n-\iota}\frac{(n-2)!(\iota-1)}{(n-\iota)!}g(\{(m_i,k_i)\}_{i=1}^s;n-\iota)\frac{R_2^{(n)}}{n^{n-l(\lambda)-2}}.
\end{split}
\end{equation}

\begin{lemma}\label{repre}
Let \[p_k=\sum_{j=1}^n x_j^k\in \mathbb{Z}[x_1, x_2, \dots x_n]\] and \[p_\lambda=\prod_{j=1}^{l(\lambda)} p_{\lambda_j},\] where $\lambda: \lambda_1\ge \lambda_2\ge \cdots\ge \lambda_{l(\lambda)}>0$  is a partition of $n$.

Then \[\prod_{j=1}^n x_j=\sum_{|\lambda|=n}(-1)^{n+l(\lambda)}\frac{1}{\prod_i k_i!m_i^{k_i}}p_\lambda.\]

In particular, \[\sum_{|\lambda|=n}(-1)^{n+l(\lambda)}\frac{n^{l(\lambda)}}{\prod_i k_i!m_i^{k_i}}=1.\]
\end{lemma}
\begin{proof}
See \cite{Man} page 35. Taking $x_j=1$ for $j=1,2, \dots, n$, we get the last identity.
\end{proof}

\begin{lemma}\label{12}
For $x\in M$, we can choose the local coordinates such that $g_{i\bar{j}}(x)=\delta_{ij}(x)$. Then
\begin{align}
&\sum_{\alpha_1,\dots, \alpha_n}\bigwedge_{i=1}^n \Theta_{\alpha_i}^{\alpha_i}=\frac{(n-1)!}{n^{n-1}}R_1^{(n)}dv,\\
&\sum_{\alpha_1,\dots, \alpha_n}\bigwedge_{j=1}^{n-2} \Theta_{\alpha_j}^{\alpha_j}\bigwedge \Theta_{\alpha_{n-1}}^{\alpha_n}\bigwedge \Theta_{\alpha_n}^{\alpha_{n-1}}=\left(\frac{(n-2)!}{n^{n-1}}R_1^{(n)}-\frac{(n-2)!}{n^{n-2}}R_2^{(n)}\right)dv,
\end{align}
where $dv=dz_{1}\wedge d\bar{z}_{1}\wedge\cdots\wedge dz_{n}\wedge d\bar{z}_{n}$.
\end{lemma}
\begin{proof}
By (\ref{RR1}), (\ref{RR2}), Lemma \ref{cn} and Lemma \ref{repre} we have
\begin{equation}\label{}
\begin{split}
 &\sum_{\alpha_1,\dots, \alpha_n}\bigwedge_{i=1}^n \Theta_{\alpha_i}^{\alpha_i}\\
=& \sum_{\alpha_1,\dots, \alpha_n}\bigwedge_{i=1}^n (\sum_{k,l}\Theta_{\alpha_ik\overline{l}}^{\alpha_i}dz_{k}\wedge d\bar{z}_l)\\
=& \sum_{\alpha_1,\dots, \alpha_n}\bigwedge_{i=1}^n (\sum_{k,l}R_{\alpha_i\overline{\alpha}_i k\overline{l}}dz_{k}\wedge d\bar{z}_l)\\
=& \sum_{\alpha_1,\dots, \alpha_n}\sum_{\sigma\in S_n} \prod_{i=1}^n R_{\alpha_i\bar{\alpha_i}i\overline{{\sigma(i)}}}dz_{1}\wedge d\bar{z}_{\sigma{(1)}}\wedge\cdots\wedge dz_{n}\wedge d\bar{z}_{\sigma{(n)}}\\
=&\sum_{|\lambda|=n}(-1)^{n+l(\lambda)}\frac{n!}{\prod_i k_i!m_i^{k_i}}\cdot\frac{R_1^{(n)}}{n^{n-l(\lambda)}} dz_{1}\wedge d\bar{z}_{1}\wedge\cdots\wedge dz_{n}\wedge d\bar{z}_{n}\\
=&\frac{(n-1)!}{n^{n-1}}R_1^{(n)}dv.
\end{split}
\end{equation}

\begin{equation}\label{conj2}
\begin{split}
&\sum_{\alpha_1,\dots, \alpha_n}\bigwedge_{i=1}^{n-2} \Theta_{\alpha_i}^{\alpha_i}\bigwedge \Theta_{\alpha_{n-1}}^{\alpha_n}\bigwedge \Theta_{\alpha_n}^{\alpha_{n-1}}\\
=&\sum_{\alpha_1,\dots, \alpha_n}\bigwedge_{i=1}^{n-2}(\sum_{k,l}\Theta_{\alpha_ik\overline{l}}^{\alpha_i}dz_{k}\wedge d\bar{z}_{l})\wedge (\sum_{k,l} \Theta_{\alpha_{n-1}k\overline{l}}^{\alpha_{n}}dz_{k}\wedge d\bar{z}_{l})\\
&\mbox\qquad\wedge(\sum_{k,l}\Theta_{\alpha_{n}k\overline{l}}^{\alpha_{n-1}}dz_{k}\wedge d\bar{z}_{l})\\
=&\sum_{\alpha_1,\dots, \alpha_n}\bigwedge_{i=1}^{n-2}(\sum_{k,l}R_{\alpha_i\overline{\alpha}_i k\overline{l}}dz_{k}\wedge d\bar{z}_{l})\wedge (\sum_{k,l} R_{\alpha_{n-1}\overline{\alpha}_{n}k\overline{l}}dz_{k}\wedge d\bar{z}_{l})\\
&\mbox\qquad\wedge(\sum_{k,l}R_{\alpha_{n}\overline{\alpha}_{n-1}k\overline{l}}dz_{k}\wedge d\bar{z}_{l})\\
=& \sum_{\alpha_1,\dots, \alpha_n}\sum_{\sigma\in S_n} \prod_{i=1}^{n-2} R_{\alpha_i\bar{\alpha_i}i\overline{{\sigma(i)}}}R_{\alpha_{n-1}\bar{\alpha}_nn-1\overline{{\sigma(n-1)}}}R_{\alpha_{n}\bar{\alpha}_{n-1}n\overline{{\sigma(n)}}}\\
&\mbox\qquad dz_{1}\wedge d\bar{z}_{\sigma{(1)}}\wedge\cdots\wedge dz_{n}\wedge d\bar{z}_{\sigma{(n)}}\\
=&(\sum_{|\lambda|=n-\iota}(-1)^{n+l(\lambda)}\frac{(n-2)!(\iota-1)}{(n-\iota)!}g(\{(m_i,k_i)\}_{i=1}^s;n-\iota)\frac{R_1^{(n)}}{n^{n-l(\lambda)-1}}\\
&\mbox \qquad  +\sum_{|\lambda|=n-\iota}(-1)^{n+l(\lambda)+1}\frac{(n-2)!(\iota-1)}{(n-\iota)!}g(\{(m_i,k_i)\}_{i=1}^s;n-\iota)\frac{R_2^{(n)}}{n^{n-l(\lambda)-2}})dv.
\end{split}
\end{equation}
Assume
\begin{equation}
\begin{split}
&\sum_{|\lambda|=n-\iota}(-1)^{n+l(\lambda)}\frac{(n-2)!(\iota-1)}{(n-\iota)!}g(\{(m_i,k_i)\}_{i=1}^s;n-\iota)\frac{R_1^{(n)}}{n^{n-l(\lambda)-1}}\\
&\mbox \qquad \qquad  +\sum_{|\lambda|=n-\iota}(-1)^{n+l(\lambda)+1}\frac{(n-2)!(\iota-1)}{(n-\iota)!}g(\{(m_i,k_i)\}_{i=1}^s;n-\iota)\frac{R_2^{(n)}}{n^{n-l(\lambda)-2}}\\
=&x_1\frac{R_1^{(n)}}{n^n}+x_2\frac{R_2^{(n)}}{n^{n-2}}.\\
\end{split}
\end{equation}
It is obviously that $x_1= -nx_2$. Moreover, by (\ref{RR2}), Lemma \ref{cn} and Lemma \ref{repre}, we have
\begin{equation}
\begin{split}
nx_1+nx_2=&\sum_{|\lambda|=n}(-1)^{n+l(\lambda)}g(\{(m_i,k_i)\}_{i=1}^s;n)\cdot\frac{R_1^{(n)}}{n^{n-l(\lambda)}}\\
=&\sum_{|\lambda|=n}(-1)^{n+l(\lambda)}\frac{n!}{\prod_i k_i!m_i^{k_i}}\cdot\frac{R_1^{(n)}}{n^{n-l(\lambda)}}\\
=&n!.
\end{split}
\end{equation}
So we solve for $x_1$ and $x_2$ and get
 \begin{equation}
\left\{ \begin{aligned}
        x_1&=n\cdot (n-2)!\\
                 x_2&= -(n-2)!.
                          \end{aligned} \right.
                          \end{equation}
Finally, we get
\begin{equation}
\begin{split}
&\sum_{\alpha_1,\dots, \alpha_n}\bigwedge_{i=1}^{n-2} \Theta_{\alpha_j}^{\alpha_j}\bigwedge \Theta_{\alpha_{n-1}}^{\alpha_n}\bigwedge \Theta_{\alpha_n}^{\alpha_{n-1}}\\
=&\left(\frac{(n-2)!}{n^{n-1}}R_1^{(n)}-\frac{(n-2)!}{n^{n-2}}R_2^{(n)}\right)dv.
\end{split}
\end{equation}
\end{proof}
Now, we can prove Theorem \ref{formula}.
\begin{proof}
Suppose the K\"{a}hler form $g_{i\bar{j}}(x)=\delta_{ij}(x)$. Then by Lemma \ref{12} and (\ref{biggerthan0}), we have
\begin{equation} \label{}
\begin{split}
&c_1^n(\Theta)-\frac{2(n+1)}{n}c_1^{n-2}(\Theta)c_2(\Theta)\\
=~&\frac{\textbf{i}^n}{(2\pi)^n}\left(\sum_{\alpha_1,\dots, \alpha_n}\bigwedge_{i=1}^n \Theta_{\alpha_i}^{\alpha_i}-\frac{n+1}{n}\left(\sum_{\alpha_1,\dots, \alpha_n}\bigwedge_{i=1}^n \Theta_{\alpha_i}^{\alpha_i}-\sum_{\alpha_1,\dots, \alpha_n}\bigwedge_{i=1}^{n-2} \Theta_{\alpha_i}^{\alpha_i}\bigwedge \Theta_{\alpha_{n-1}}^{\alpha_n}\bigwedge \Theta_{\alpha_n}^{\alpha_{n-1}}\right)\right)\\
=~&\frac{\textbf{i}^n}{(2\pi)^n}\left(-\frac{1}{n}\sum_{\alpha_1,\dots, \alpha_n}\bigwedge_{i=1}^n \Theta_{\alpha_i}^{\alpha_i}+\frac{n+1}{n}\sum_{\alpha_1,\dots, \alpha_n}\bigwedge_{i=1}^{n-2} \Theta_{\alpha_i}^{\alpha_i}\bigwedge \Theta_{\alpha_{n-1}}^{\alpha_n}\bigwedge \Theta_{\alpha_n}^{\alpha_{n-1}}\right)\\
=~&\frac{\textbf{i}^n}{(2\pi)^n}\left(-\frac{1}{n}\frac{(n-1)!}{n^{n-1}}R_1^{(n)}+\frac{n+1}{n}\left(\frac{(n-2)!}{n^{n-1}}R_1^{(n)}-\frac{(n-2)!}{n^{n-2}}R_2^{(n)}\right)\right)dv\\
=~&\frac{\textbf{i}^n}{(2\pi)^n}\left(\frac{2(n-2)!}{n^n}R_1^{(n)}-\frac{(n+1)(n-2)!}{n^{n-1}}R_2^{(n)}\right)dv\\
=~&\frac{\textbf{i}^n}{(2\pi)^n}\frac{(n-2)!}{n^n}(2R_1^{(n)}-n(n+1)R_2^{(n)})dv\\
=~&-\frac{\textbf{i}^n}{(2\pi)^n}\frac{(n+1)^{n}(n+2)(n+3)(n-2)!}{2^{n}\cdot V}\int (R_{\xi\bar{\xi}\xi\bar{\xi}}-Ave(R))^2 Ave(R)^{n-2}d\mathbb{S}^{2n-1}\cdot dv.
 \end{split}
 \end{equation}
So integrating both sides over $M$, we have
 \begin{equation} \label{}
\begin{split}
&-c_1^n(M)+\frac{2(n+1)}{n}c_2(M)c_1^{n-2}(M) \\
=~&\frac{\textbf{i}^n}{(2\pi)^n}\frac{(n+1)^{n-1}(n+3)!}{2^{n}n(n-1)\cdot V}\int_M\left(\int (R_{\xi\bar{\xi}\xi\bar{\xi}}-Ave(R))^2 Ave(R)^{n-2}d\mathbb{S}^{2n-1}\right)dv\\
=~&\frac{(n+1)^{n-1}(n+3)!}{(2\pi)^{n}n(n-1)\cdot V}\cdot\left(\frac{2\mu}{n+1}\right)^{n-2}\cdot\left(\frac{\textbf{i}}{2}\right)^n\int_M\left(\int (R_{\xi\bar{\xi}\xi\bar{\xi}}-Ave(R))^2 d\mathbb{S}^{2n-1}\right)dv\\
=~&\frac{(n+1)\cdot(n+3)!\mu^{n-2}}{4(\pi)^{n}n(n-1)\cdot V}\cdot\int_M\left(\int (R_{\xi\bar{\xi}\xi\bar{\xi}}-Ave(R))^2 d\mathbb{S}^{2n-1}\right)d\tilde{v},
 \end{split}
 \end{equation}
 where the second equation holds by Berger’s lemma, namely $Ave(R)=\frac{2\mu}{n+1}$ or from (\ref{AveR}) and (\ref{Berger}).
\end{proof}

\begin{remark}
Yau's celebrated solution to the Calabi's conjecture (\cite{Ya1})
says that a compact K\"{a}hler $n$-fold $M$ with negative or
trivial tangent bundle always admit K\"{a}hler-Einstein
metrics. As an application, he proved in \cite{Ya2} that for the
case $c_1(M)<0$, there is a Chern number inequality:
\[(-1)^n(2(n+1)c_2(M)c_1^{n-2}(M)-nc_1^n(M))\ge 0,\]
with equality hold if and only if $M$ is uniformized by the unit
ball in $\mathbb{C}^n$.
From this theorem, we can get Yau's famous inequality and characterise the equality by using \cite{Ko-No} Theorem 7.5.
\end{remark}

From Theorem \ref{formula}, we can have a reverse Yau's inequality (Theorem \ref{rY}).
\begin{proof}
Suppose that the holomorphic sectional curvature at is nonpositive. We assume $Ric_{i\bar{j}}=-g_{i\bar{j}}$ without loss of generality.
\begin{equation}\label{rev}
\begin{split}
& -n|c_1^{n}(M)|+2(n+1)c_2(M)|c_1^{n-2}(M)|\\
=&\frac{(n+1)\cdot(n+3)!}{4(\pi)^{n}(n-1)\cdot V}\cdot\int_M\left(\int (R_{\xi\bar{\xi}\xi\bar{\xi}}-Ave(R))^2 d\mathbb{S}^{2n-1}\right)d\tilde{v}\\
=&\frac{(n+1)\cdot(n+3)!H_m^2}{4(\pi)^{n}(n-1)\cdot V}\cdot\int_M\left(\int \left(\frac{R_{\xi\bar{\xi}\xi\bar{\xi}}}{H_m}-\frac{Ave(R)}{H_m}\right)^2 d\mathbb{S}^{2n-1}\right)d\tilde{v}\\
=&\frac{(n+1)\cdot(n+3)!H_m^2}{4(\pi)^{n}(n-1)}\cdot\int_M\left(\frac{\int\left(\frac{R_{\xi\bar{\xi}\xi\bar{\xi}}}{H_m}+\frac{1}{2}\right)^2 d\mathbb{S}^{2n-1} }{V}-\left(\frac{Ave(R)}{H_m}+\frac{1}{2}\right)^2\right)d\tilde{v}\\
\le& \frac{(n+1)\cdot(n+3)!H_m^2}{4(\pi)^{n}(n-1)}\cdot\int_M\left(\frac{1}{4}-\left(\frac{Ave(R)}{H_m}+\frac{1}{2}\right)^2\right)d\tilde{v}\\
=&\frac{(n+1)\cdot(n+3)!H_m^2}{4(\pi)^{n}(n-1)}\cdot\int_M\left(1-\frac{|Ave(R)|}{H_m}\right)\left(\frac{|Ave(R)|}{H_m}\right)d\tilde{v}\\
\le& \frac{|c_1^{n}(M)|(n+1)^2(n+2)(n+3)H_m^2}{16(n-1)}.
\end{split}
\end{equation}
In particular, $H_m\le 1$ since $Ric_{i\bar{j}}=-g_{i\bar{j}}$. So the special case follows obviously.

The equality of the last inequality in the proof can hold if and only if $$Ave(R)=-\frac{1}{2}H_m.$$ The equality of the first inequality holds if and only if for almost every $x\in M$, $R_{\xi\bar{\xi}\xi\bar{\xi}}=0$ or $-H_m$ almost everywhere.
Suppose \[S_0=\{\xi\in
\mathbb{S}^{2n-1}\subseteq T_x(M)~|~R_{\xi\bar{\xi}\xi\bar{\xi}}=0\},\]
\[S_1=\{\xi\in
\mathbb{S}^{2n-1}\subseteq T_x(M)~|~R_{\xi\bar{\xi}\xi\bar{\xi}}=-H_m\}\] and
\[S_2=\{\xi\in
\mathbb{S}^{2n-1}\subseteq T_x(M)~|~-H_m<R_{\xi\bar{\xi}\xi\bar{\xi}}<0\}.\]
Denote vol$(S_i)$ to be the volume of $S_i$, where $i=1,2,3$.
Since there is no compact K\"{a}hler-Einstein manifold $M$ such that at almost all the point $x\in M$,  vol$(S_0)=vol(S_1)=\frac{1}{2}V$ and vol$(S_2)=0$, the inequality in (\ref{rev}) is not sharp.

If the holomorphic sectional curvature is nonnegative, the arguments are similar.
\end{proof}


Similarly, we can proof Theorem \ref{ampleK}.
\begin{proof}

For $n\ge 2$, by Theorem \ref{formula}
\begin{equation}
\begin{split}
& -n|c_1^{n}(M)|+2(n+1)c_2(M)|c_1^{n-2}(M)|\\
=&\frac{(n+1)\cdot(n+3)!}{4(\pi)^{n}(n-1)\cdot V}\cdot\int_M\left(\int (R_{\xi\bar{\xi}\xi\bar{\xi}}-Ave(R))^2 d\mathbb{S}^{2n-1}\right)d\tilde{v}\\
=&\frac{(n+1)\cdot(n+3)!H_m^2}{4(\pi)^{n}(n-1)}\cdot\int_M\left(\frac{\int \left(\frac{R_{\xi\bar{\xi}\xi\bar{\xi}}}{H_m}\right)^2 d\mathbb{S}^{2n-1}}{V}-\left(\frac{Ave(R)}{H_m}\right)^2\right)d\tilde{v}\\
\le& \frac{(n+1)\cdot(n+3)!H_m^2}{4(\pi)^{n}(n-1)}\cdot\int_M\left(1-\left(\frac{Ave(R)}{H_m}\right)^2\right)d\tilde{v}\\
\le &\frac{(n+1)\cdot(n+3)!H_m^2}{4(\pi)^{n}(n-1)}\cdot V(M)\\
=&\frac{(n+1)^2(n+2)(n+3)H_m^2}{4(n-1)}|c_1^{n}(M)|
\end{split}
\end{equation}
So
\begin{equation}\label{nge2}
2(n+1)c_2(M)|c_1^{n-2}(M)| \le \left(\frac{(n+1)^2(n+2)(n+3)H_m^2}{4(n-1)}+n\right)|c_1^n(M)|.
\end{equation}
The equality of inequality (\ref{nge2}) holds if and only if for almost every point $x\in M$, $R_{\xi\bar{\xi}\xi\bar{\xi}}=\pm H_m$ for almost all $\xi$ and $Ave(R)=0$, which is impossible. So the above inequality is not sharp.
\end{proof}


\section{\textbf{pinched theorem}}
Let $\delta$ be the pinching constant. Then $0<\delta\le 1$.

As a corollary of Theorem \ref{formula}, we can get Theorem \ref{pinch}.
%
\begin{proof}
\begin{enumerate}
\item [1)] If $M$ is with positive holomorphic sectional curvature and $Ric_{i\bar{j}}=\mu g_{i\bar{j}}$ ($\mu>0$), then
\begin{equation}
\begin{split}
0\le& -nc_1^{n}(M)+2(n+1)c_2(M)c_1^{n-2}(M)\\
=&\frac{(n+1)\cdot(n+3)!\mu^{n-2}}{4(\pi)^{n}(n-1)\cdot V}\cdot\int_M\left(\int (R_{\xi\bar{\xi}\xi\bar{\xi}}-Ave(R))^2 d\mathbb{S}^{2n-1}\right)d\tilde{v}\\
\le&\frac{(n+1)\cdot(n+3)!\mu^{n-2}H_m^2}{4(\pi)^{n}(n-1)}\cdot\int_M\int (1-\delta)^2d\tilde{v}\\
=&\frac{(n+1)\cdot(n+3)!\mu^{n-2}H_m^2}{4(\pi)^{n}(n-1)}\cdot (1-\delta)^2\cdot V(M)\\
=&\frac{c_1^n(M)(n+1)^{2}(n+2)(n+3)H_m^2}{4(n-1)\mu^2}\cdot (1-\delta)^2\\
\le & \frac{c_1^n(M)(n+1)^{2}(n+2)(n+3)}{4(n-1)}\cdot (1-\delta)^2.
\end{split}
\end{equation}
If
\begin{equation}
\delta>1-(\frac{4(n-1)}{c_1^n(M)(n+1)^{2}(n+2)(n+3)})^{1/2},
\end{equation}
then \[-nc_1(M)^{n}+2(n+1)c_2(M)c_1(M)^{n-2}<1.\]
Therefore \[-nc_1(M)^{n}+2(n+1)c_2(M)c_1(M)^{n-2}=0\] and $M$ is with constant holomorphic sectional curvature (\cite{Ko-No} Theorem 7.5). So $M$ is holomorphically isometric to $\mathbb{P}^n$.
\item [2)] If $M$ is with negative holomorphic sectional curvature and $Ric_{i\bar{j}}=-\mu g_{i\bar{j}}$ ($\mu>0$), then similarly,
\begin{equation}
\begin{split}
0\le& (-1)^n(-nc_1(M)^{n}+2(n+1)c_2(M)c_1(M)^{n-2})\\
=&\frac{(-1)^n(n+1)\cdot(n+3)!\mu^{n-2}}{4(\pi)^{n}(n-1)\cdot V}\cdot\int_M\left(\int (R_{\xi\bar{\xi}\xi\bar{\xi}}-Ave(R))^2 d\mathbb{S}^{2n-1}\right)d\tilde{v}\\
\le & \frac{|c_1^n(M)|(n+1)^{2}(n+2)(n+3)}{4(n-1)}\cdot (1-\delta)^2.
\end{split}
\end{equation}
If
\begin{equation}
\delta>1-(\frac{4(n-1)}{|c_1^n(M)|(n+1)^{2}(n+2)(n+3)})^{1/2},
\end{equation}
then \[-nc_1(M)^{n}+2(n+1)c_2(M)c_1(M)^{n-2}=0\] and $M$ is with constant holomorphic sectional curvature  (\cite{Ko-No} Theorem 7.5). So $M$ a is holomorphically isometric to ball quotient.
\item [3)]If $n\ge 2$ and $\mu\neq 0$, then
\begin{equation}
\begin{split}
0\le &|-nc_1(M)^{n}+2(n+1)c_2(M)c_1(M)^{n-2}|\\
=&\frac{(n+1)\cdot(n+3)!|\mu|^{n-2}}{4(\pi)^{n}(n-1)\cdot V}\cdot\int_M\left(\int (R_{\xi\bar{\xi}\xi\bar{\xi}}-Ave(R))^2 d\mathbb{S}^{2n-1}\right)d\tilde{v}\\
\le &\frac{(n+1)\cdot(n+3)!|\mu|^{n-2}}{4(\pi)^{n}(n-1)\cdot V}\cdot\int_M\left(\int a^2 d\mathbb{S}^{2n-1}\right)d\tilde{v} \\
=&\frac{a^2(n+1)\cdot(n+3)!|\mu|^{n-2}}{4(\pi)^{n}(n-1)}V(M)\\
=&\frac{a^2(n+1)^2(n+2)(n+3)}{4(n-1)\mu^2}|c_1(M)^{n}|<1.
\end{split}
\end{equation}
%

So $R_{\xi\bar{\xi}\xi\bar{\xi}}=Ave(R)$. Then $M$ is with constant holomorphic sectional curvature $0$ (\cite{Ko-No} Theorem 7.5). Then $M$ is holomorphically isometric to a complex torus.
\end{enumerate}
\end{proof}

\begin{remark}
If $M$ is (negative) $\delta$-holomorphically pinched, then it is (negative) $\frac{1}{4}(3\delta-2)$-pinched. If $M$ is (negative) $\delta$-pinched, then it is (negative) $\frac{\delta(8\delta+1)}{1-\delta}$-holomorphically pinched (see \cite{Ber1} Proposition 1 and \cite{Bi-Go} Section 8). If $M$ is a K\"{a}hler-Einstein manifold, then we improve the well known $\frac{1}{4}$-pinched theorem (see \cite{Ko-No} page 369) and negative $\frac{1}{4}$-pinched theorem by Yau and Zheng (\cite{Ya-Zh}).
\end{remark}

Similarly, we can prove Theorem \ref{small}.
\begin{proof}
\begin{enumerate}
\item [1)] If $M$ is with positive holomorphic sectional curvature, then using the same arguments in the proof of Theorem \ref{rY}, we have
\begin{equation}
\begin{split}
0\le& -nc_1^{n}(M)+2(n+1)c_2(M)c_1^{n-2}(M)\\
=&\frac{(n+1)\cdot(n+3)!\mu^{n-2}}{4(\pi)^{n}(n-1)\cdot V}\cdot\int_M\left(\int (R_{\xi\bar{\xi}\xi\bar{\xi}}-Ave(R))^2 d\mathbb{S}^{2n-1}\right)d\tilde{v}\\
<& \frac{c_1^{n}(M)(n+1)^2(n+2)(n+3)H_m^2}{16(n-1)}\le 1.
\end{split}
\end{equation}

Therefore \[-nc_1(M)^{n}+2(n+1)c_2(M)c_1(M)^{n-2}=0\] and $M$ is holomorphically isometric to $\mathbb{P}^n$.
\item [2)] If $M$ is with negative holomorphic sectional curvature, then similarly,
\begin{equation}
\begin{split}
0\le& (-1)^n(-nc_1(M)^{n}+2(n+1)c_2(M)c_1(M)^{n-2})\\
=&\frac{(-1)^n(n+1)\cdot(n+3)!\mu^{n-2}}{4(\pi)^{n}(n-1)\cdot V}\cdot\int_M\left(\int (R_{\xi\bar{\xi}\xi\bar{\xi}}-Ave(R))^2 d\mathbb{S}^{2n-1}\right)d\tilde{v}\\
<& \frac{|c_1^{n}(M)|(n+1)^2(n+2)(n+3)H_m^2}{16(n-1)}\le 1.
\end{split}
\end{equation}
So \[-nc_1(M)^{n}+2(n+1)c_2(M)c_1(M)^{n-2}=0\] and $M$ is holomorphically isometric to a ball quotient.
\end{enumerate}
\end{proof}

\section{\textbf{Negativity of the holomorphic sectional curvature and positivity of the canonical bundle}}
In 2016, Wu-Yau (\cite{Wu-Ya}) showed that if $M$ is a smooth projective manifold with negative holomorphic sectional curvature everywhere, then $M$ has a K\"{a}hler metric whose Ricci curvature is everywhere negative. In particular, the canonical bundle $K_M$ is ample. Later, Tosatti-Yang (\cite{To-Ya}) generalized Wu-Yau's result to compact K\"{a}hler manifolds. Conversely, however, the result doesn't hold. In fact, Demailly (\cite{Dem}) constructed a smooth projective surface which is Kobayashi hyperbolic (with ample canonical bundle), but the projective surface cannot admit any hermitian metric with negative holomorphic sectional curvature. In general, a Fermat hypersurface of degree $N+2$ in the projective space of dimension $N$ contains many rational curves and has ample canonical line bundle. However, it is well known that a Hermitian manifold with negative holomorphic sectional curvature doesn't admit any rational curves (cf. \cite{Shi} or Lemma 3.1 in \cite{H-L-W} for more general argument). Usually, for a higher dimensional projective manifold with non-nef anti-canonical divisor, it is so easy to determine if there are rational curves on them or not.

We use covering theory to construct a projective manifold of dimension $n$ $(n\ge 2)$ with ample canonical bundle which doesn't carry any hermitian metric with negative holomorphic sectional curvature by Theorem \ref{rY}.  So we need Izawa's result for calculating the Chern numbers of the covering spaces (\cite{Iza}).

Let $f: Y\rightarrow X$ be a ramified covering between $n$-dimensional compact complex manifolds with covering multiplicity $u$. Let $R_f=\sum_i r_iR_i$ be the ramification divisor of $f$, and $B_f=\sum_i b_iB_i$ the branch locus of $f$. We set $f^*B_i=\sum_t r_{i_t}R_{i_t}$ where the induced map $f|_{R_{i_t}}: R_{i_t}\rightarrow B_i$ with mapping degree $n_{i_t}$. So $b_i=\sum_t n_{i_t}r_{i_t}$. We assume that the ramification divisor and the irreducible component of the branch locus are all nonsingular. Then

\begin{equation}\label{coverChern}
\begin{split}
 &c_1^{N_1}\cdots c_n^{N_n}(Y)-u c_1^{N_1}\cdots c_n^{N_n}(X)\\
=&\sum_i\sum_{\alpha=0}^{n-1}(\sum_t\frac{n_{i_t}(1-(r_{i_t}+1)^{\alpha+1})}{(r_{i_t}+1)^\alpha})P_{\alpha}(c_1(B_i), \cdots, c_{n-1}(B_i))\cdot c_1(L_{B_i})^\alpha\cap[B_i],
\end{split}
\end{equation}
where $L_{B_i}$ is the line bundle corresponding to the divisor $B_i$, $P_\alpha$ is the coefficient of $l^\alpha$ in
\[H_{\xi}^{(N_1\cdots N_n)}(l):=l^{-1}\cdot((\prod_{i=1}^n(c_i(\xi)+c_{i-1}(\xi)\cdot l)^{N_i})-c_1^{N_1}\cdot c_n^{N_n}(\xi))\] as a polynomial in $l$, i.e. \[H_{\xi}^{(N_1\cdots N_n)}(l)=\sum_{\alpha=0}^{n-1} P_{\alpha}(c_1,\cdots, c_{n-1})l^\alpha.\]

\begin{example}\label{example dimn}
Let $f: M\rightarrow \mathbb{P}^n$ be the double covering whose branch locus $B$ is a smooth hypersurface of degree $d=2n+4$ ($n\ge 2$). We will calculate $c_1^n(M)$ and $c_1^{n-2}c_2(M)$. Then we  need to show that they don't satisfy the inequality in Theorem \ref{rY}.

Let $H$ be a hyperplane in $\mathbb{P}^n$. Then $K_M=f^*(K_{\mathbb{P}^n}+(n+2)H)=f^*(H)$. So for any curve $C$ in $M$, $K_M\cdot C>0$. Moreover, $K_M^n=(f^*(H))^n=2$. Therefore $K_M$ is ample and $c_1^n(M)=(-1)^n2$.
We can also calculate the Chern classes of the tangent bundle to the smooth hypersurface $B$ of degree $d$. In fact, we have the standard tangent bundle sequence
\[0\rightarrow T_B\rightarrow T_{\mathbb{P}^n}|_B\rightarrow N_{B/\mathbb{P}^n}\rightarrow 0,\]
where the normal bundle $N_{B/\mathbb{P}^n}=\mathcal{O}_B(d)$.
Denote $\zeta_B:=H|_{B}$. Using Whitney's formula, we have
\begin{equation*}
\begin{split}
c(B)=&c(T_B)=\frac{c(T_{\mathbb{P}^n}|_B)}{N_{B/\mathbb{P}^n}}=\frac{(1+\zeta_B)^{n+1}}{1+d\zeta_B}\\
=&(1+(n+1)\zeta_B+\frac{n(n+1)}{2}\zeta_B^2+\cdots)(1-d\zeta_B+d^2\zeta_B^2+\cdots).
\end{split}
\end{equation*}
So
\begin{equation*}
\begin{split}
c_1(B)=&(n+1-d)\zeta_B,\\
c_2(B)=&(d^2-(n+1)d+\frac{n(n+1)}{2})\zeta_B^2.
\end{split}
\end{equation*}
By (\ref{coverChern}), we have
\begin{equation*}
\begin{split}
 &c_1^{n-2}c_2(M)-2c_1^{n-2}c_2(\mathbb{P}^n)\\
=&c_1^{n-2}c_2(M)-n(n+1)^{n-1}\\
=&(\frac{dH}{2})^{-1}((c_1(B)+\frac{dH}{2})^{n-2}(c_2(B)+c_1(B)\cdot \frac{dH}{2})-c_1^{n-2}c_2(B))\cap [B]\\
&-2(dH)^{-1}((c_1(B)+dH)^{n-2}(c_2(B)+c_1(B)\cdot dH)-c_1^{n-2}c_2(B))\cap [B]\\
=&(\frac{d}{2})^{-1}(((n+1-d)+\frac{d}{2})^{n-2}(d^2-(n+1)d+\frac{n(n+1)}{2}+(n+1-d)\cdot \frac{d}{2})\\
&-(n+1-d)^{n-2}(d^2-(n+1)d+\frac{n(n+1)}{2}))d\\
&-2(d)^{-1}(((n+1-d)+d)^{n-2}(d^2-(n+1)d+\frac{n(n+1)}{2}+(n+1-d)d\\
&-(n+1-d)^{n-2}(d^2-(n+1)d+\frac{n(n+1)}{2}))d\\
=&2(((n+1-\frac{d}{2})^{n-2}(d^2-(n+1)d+\frac{n(n+1)}{2}+\frac{d(n+1-d)}{2})\\
&-(n+1)^{n-2}(d^2-(n+1)d+\frac{n(n+1)}{2}+d(n+1-d)))\\
=&(-1)^{n-2}(3n^2+11n+12)-n(n+1)^{n-1},
\end{split}
\end{equation*}
i.e. \[c_1^{n-2}c_2(M)=(-1)^{n-2}(3n^2+11n+12).\]
For $n\ge 2$,  the inequality in Theorem \ref{rY} doesn't hold.
So the examples we construct are with ample canonical bundles, but don't carry any hermitian metric with negative holomorphic sectional curvatures.
\end{example}
\section*{Acknowledgements}
The author would like to thank professor Ngaiming Mok for discussing about
this problem and supporting his research when he was in the
University of Hong Kong several years ago, and thank professor Gang Liu for useful discussions and notifying him the interesting paper by Siu and Yang (\cite{Si-Ya}). The author also would like to thank the anonymous referee for many useful suggestions to simplify his formulas in the original version.

\nocite{*}
\bibliographystyle{amsalpha}
\bibliography{Chern-reference}

\end{document}